\documentclass[a4paper,12pt]{amsart}
 
\usepackage{amssymb}
\usepackage[latin1]{inputenc}
\usepackage{url,xcolor}
  \usepackage[left=2cm,right=2cm, top=3cm, bottom=3cm]{geometry}

\usepackage{hyperref} %

  \usepackage[pagewise]{lineno}

\usepackage{listings} 
 
 \lstdefinelanguage{Magma}%
  {%
   otherkeywords={:=,+:=,-:=,*:=},%
   procnamekeys={function,func,intrinsic,procedure,proc,return},%
   morekeywords={true,false},%
   morekeywords=[2]{adj,and,cat,cmpeq,cmpne,diff,div,eq,ge,gt,in,is,join,le,lt,%
          meet,mod,ne,notadj,notin,notsubset,or,sdiff,subset,xor},%
   morekeywords=[3]{assigned,break,by,case,catch,continue,declare,default,%
          delete,do,elif,else,end,eval,exists,exit,for,forall,fprintf,if,local,%
          not,print,printf,quit,random,read,readi,repeat,restore,save,select,%
          then,time,to,try,until,vprint,vprintf,vtime,when,where,while},%
   morekeywords=[4]{clear,forward,freeze,iload,import,load},%
   morekeywords=[5]{assert,assert2,assert3,error,require,requirege,requirerange},%
   morekeywords=[6]{car,comp,cop,elt,ext,frac,hom,ideal,iso,lideal,loc,map,%
          ncl,pmap,quo,rec,recformat,rep,rideal,sub},%
      sensitive,%
      morecomment=[l]//
  }[keywords,procnames,comments,strings]%

\lstnewenvironment{code_magma}[1][]
{\lstset{basicstyle=\scriptsize\ttfamily, columns=fullflexible,
language=Magma,
keywordstyle=\color{red}\bfseries,
commentstyle=\color{blue},tabsize=4,
numbers=left, numberstyle=\tiny,
stepnumber=2, numbersep=5pt, frame=single, #1}}{} 
 
\newtheorem{theorem}{Theorem}[section]
\newtheorem*{thm}{Theorem}
\newtheorem{lemma}[theorem]{Lemma}
\newtheorem{proposition}[theorem]{Proposition}

\newtheorem{corollary}[theorem]{Corollary}

\theoremstyle{definition}     
\newtheorem{definition}[theorem]{Definition}

\theoremstyle{remark}
\newtheorem{remark}[theorem]{Remark}
\numberwithin{equation}{section}

\newcommand\Bs{{\mathrm{Bs}}}
\newcommand\Cl{{\mathrm{Cl}}}
\newcommand\Pic{{\mathrm{Pic}}}
\newcommand\N{{\mathrm{N}}}
\newcommand\Eff{{\mathrm{Eff}}}

\newcommand\Nef{{\mathrm{Nef}}}
\newcommand\Mov{{\mathrm{Mov}}}
\newcommand\Cox{{\mathrm{Cox}}}
\newcommand\SAmp{{\mathrm{SAmp}}}
\newcommand\Aut{{\mathrm{Aut}}}
\newcommand{\Stab}{\mathrm{Stab}}

\newcommand{\ord}{\text{ord}}
\newcommand{\Jac}{\mathrm{Jac}}

\newcommand{\Fix}{\mathrm{Fix}}

\newcommand{\Cone}{\mathbb R_{\geq 0}}

\title[Divisors on surfaces isogenous to a product of mixed type]{Divisors on surfaces isogenous to \\ a product of mixed type with $p_g=0$}

\makeatletter

\@addtoreset{equation}{section}
\makeatother

\author[D. Frapporti]{Davide Frapporti}
\address{University of Bayreuth, Lehrstuhl Mathematik VIII; 
Universit\"atsstra\ss e 30, D-95447 Bayreuth, Germany}
\email{Davide.Frapporti@uni-bayreuth.de}

\author[K.-S. Lee]{Kyoung-Seog Lee}
\address{Center for Geometry and Physics, Institute for Basic Science (IBS), Pohang 37673, Republic of Korea}
\address{{\it Current Address}: Institute of the Mathematical Sciences of the Americas, University of Miami, 1365 Memorial Drive, Ungar 515, Coral Gables, FL 33146, USA}
\email{kyoungseog02@gmail.com}


\thanks{The first author is member of G.N.S.A.G.A. of I.N.d.A.M. and acknowledges support of the ERC-advanced Grant 340258-TADMICAMT. The second author was supported by IBS-R003-Y1.}

\date{\today}

\begin{document}
\begin{abstract} 
In this paper, we study effective, nef and semiample cones of surfaces isogenous to a product of mixed type with $p_g=0$. In particular, we prove that all reducible fake quadrics are Mori dream surfaces.
\end{abstract}
\maketitle

\section*{Introduction}
Effective, nef, semiample cones and Cox rings of algebraic varieties are fundamental invariants in birational geometry.  In \cite{HK00}
Hu and Keel proved that varieties having finitely generated Cox ring are ideal objects in birational geometry. They called those varieties \textit{Mori dream spaces}. After their work there have been intensive investigations about Cox rings and Mori dream spaces, see \cite{ADHL15} and references therein for more details. Most known examples of Mori dream spaces have non-positive Kodaira dimensions. However, the theory of Cox rings can be applied to wider classes of varieties, but Cox rings of varieties having positive Kodaira dimensions have been far less studied. Especially, we do not have many examples of Mori dream spaces of general type except some trivial cases. For example, a normal projective $\mathbb{Q}$-factorial variety with $h^1(\mathcal O)=0$ and Picard number one is a Mori dream space. In dimension 2, it is natural  to look for examples among minimal smooth projective surfaces of general type with $p_g=0$, since they have automatically $q=0$, and relatively small Picard numbers. 
A minimal smooth projective surface of general type with $p_g=0$ and Picard number 1 is called a \textit{fake projective plane}. It turns out that there are exactly 100 fake projective planes (cf. \cite{CS10, PY07, PY10}) and all of them provide examples of Mori dream surfaces of general type. Keum and the second author proved that there are more examples of Mori dream surfaces of general type with $p_g=0$ whose Picard numbers are greater than 1 (cf. \cite{KL19}). However, we do not know whether all minimal smooth projective surfaces with $p_g=0$ are Mori dream surfaces or not.

 In this view point, fake quadrics are interesting objects since they have Picard number two (hence become the next test case),  but their geometry is rather mysterious so far. 
All known examples of fake quadrics can be divided into two classes: reducible or irreducible fake quadrics. Reducible fake quadrics are \textit{surfaces isogenous to a product}:
a surface  $S$ is isogenous to a product (of curves)
if there exist compact Riemann surfaces $C_1, C_2$ of genus $g(C_i)\geq 2$  and 
a finite group  $G \leq \Aut(C_1 \times  C_2) $  acting freely on the product $C_1 \times C_2$ such that
$S=(C_1 \times C_2)/G$ (see \cite{Cat00}).
In the recent years there has been intensive works on those surfaces 
birational to the quotient of the product of two curves by the action of  a finite group. We refer to \cite{CF18, PignaQE} and the references therein for a recent and detailed account on it.

Surfaces isogenous to a product  divide naturally into two types:
\textit{unmixed} if $G$ acts diagonally on the product, \textit{mixed} if there are elements of $G$ which exchange the two factors.
Every surface isogenous to a product of unmixed type with $p_g=0$ carries two natural fibrations onto $\mathbb P^1$,
and one can use these fibrations to show that these surfaces are Mori dream surfaces, see \cite{KL19}.

In this paper, we study effective, nef and semiample cones and finite generation of Cox rings of surfaces isogenous to a product of mixed type with $p_g=0.$ These surfaces have been completely classified in \cite{BCG08,Frap13} and form 5 irreducible connected components in the moduli space of minimal smooth complex projective surfaces with $\chi=1$ and $K^2=8$.
We use their explicit group theoretic description and  the geometric structure arising from the construction to obtain the following result.

\begin{thm}
Let $S$ be a surface isogenous to a product of mixed type with $p_g=0.$ Then $S$ is a Mori dream surface and $\Eff(S)=\Nef(S)=\SAmp(S) $. 
\end{thm}

Therefore we prove that all reducible fake quadrics are Mori dream surfaces. We raise the question whether every fake quadric is a Mori dream surface or not.

Let us briefly explain our strategy to study geometry and divisors of surfaces isogenous to a product of mixed type. Let $S=(C \times C)/G$ be a surface isogenous to a product of mixed type. Studying divisors on $S$ is equivalent to studying $G$-invariant divisors on $C \times C.$ Therefore, we want to construct several $G$-invariant divisors on $C \times C$ in order to get divisors lying on extremal rays of effective, nef and semiample cones of $S.$ To construct $G$-invariant divisors on $C \times C$ in a systematic way, we will use the following approach.
 Let $H$ be a subgroup of $\Aut(C).$ For each $h \in H$, let us consider the graph of the $h$-action $\Gamma_h = \{ (x,hx) | x \in C \}$. If $G^0<H$ (see Definition \ref{def_iso}), then the group $G$ acts on the set  $\{ \Gamma_h | h \in H \}$, and from  the $G$-orbits of this set of graphs, we  obtain several $G$-invariant divisors on $C \times C$, which descend to effective divisors on $S$.  We call these divisors \textit{orbit divisors} induced by $H$.
 Because we can compute the intersection numbers of these $G$-invariant divisors on $C \times C,$ we can compute intersection numbers of the orbit divisors and then determine  the location of the orbit divisors in $\N^1(S).$ By analysing the orbit divisors and their intersections case by case, we obtain the above result. 
 When the curve $C$ is a covering of $\mathbb P^1$ branched over 5 points (see Theorem \ref{thm:classification}), taking $H=G^0$ is enough to construct extremal rays of $\Eff(S), \Nef(S)$ and $\SAmp(S)$.
In the cases where the curve $C$ is a covering of $\mathbb P^1$ branched over 3 points, it is not enough to consider the automorphisms in $G^0$, because we get only one vector in $\N^1(S)$.
Our idea is then to  lift the symmetries of the configuration of 3 points in $\mathbb P^1$ 
and show  that $\Aut(C)$ contains a subgroup $H$  bigger than $G^0$.

The paper is organized as follows: in Section 1 we recall basic definitions and facts on  effective, nef, semiample cones and Mori dream spaces. In Section 2 we recall the definition of surfaces isogenous to a product and the above mentioned classification. Section 3 is devoted to explaining the construction of orbit divisors and to the proof of the main theorem.

\noindent {\bf Notation.} We work over $\mathbb{C}$
and use the standard notation in surface theory:
for a smooth projective  surface $S$ we denote by $p_g(S):=h^2(S, \mathcal O_S)$ its
geometric genus, by $q(S):= h^1(S, \mathcal O_S)$ its irregularity, 
by $\chi(\mathcal{O}_S)=1-q(S) +p_g(S)$ its holomorphic Euler-Poincar\'e characteristic, 
and by $K^2_S$ the self-intersection of its  canonical divisor.
For a smooth projective curve (Riemann surface)  $C$ we denote by $g(C)$  its genus.

We write $D_1 \sim_{num} D_2$ to denote that the divisors $D_1$ and $D_2$ are numerically equivalent. 

We denote by $\mathbb  Z_n$ the cyclic group of order $n$, 
by $D_n$ the 	dihedral group of order $2n$,
 and by	 $D_{p,q,r}$ the group $\langle x,y\mid x^p= y^q=1, xyx^{-1}=y^r\rangle$.
When $G$ is an abelian group, then $G_{\mathbb{R}}$(resp. $G_{\mathbb{Q}}$) denotes $G \otimes_{\mathbb{Z}} {\mathbb{R}}$(resp. $G \otimes_{\mathbb{Z}} {\mathbb{Q}}$). 

The rest of the notation is standard in algebraic geometry.

\bigskip

{\bf Acknowledgements.} 
The authors thank Stephen Coughlan for inspiring conversations and a careful reading of the paper, 
and  Ingrid Bauer and Fabrizio Catanese for helpful discussions and suggestions.
The second author thanks JongHae Keum for helpful discussions and suggestions.

\section{Mori Dream Surfaces}

In this section we collect several definitions and facts about effective, nef and semiample cones of algebraic surfaces.

\begin{definition} Let $X$ be a normal projective variety and $D$ be a Weil divisor on $X$.
We denote by $ \Bs|D|$ the base locus of $|D|$,  and by $\mathbb B |D|:=\bigcap_{n=1}^\infty\Bs|nD|$ the
 stable base locus of $|D|$.
\begin{enumerate}
\item The divisor $D$  is \textit{semiample} if $\Bs |nD|= \emptyset$ for some  positive $n\in \mathbb Z$.
The \textit{semiample cone} $\SAmp(X) \subset \N^1(X)(=\Cl(X)_{\mathbb R}/\sim_{num})$ is the convex cone generated by semiample divisors. 

\item The divisor $D$ is \textit{movable} if $\mathbb B |D|$  has codimension at  least 2.
The\textit{ movable cone} $\Mov(X)\subset \N^1(X)$ is the convex cone generated by movable divisors. 

\item The \textit{nef cone}  $\Nef(X)\subset \N^1(X)$ is the convex cone generated by nef divisors.

\item The e\textit{ffective cone}  $\Eff(X) \subset \N^1(X) $ is the convex cone generated  by effective divisors.

\end{enumerate}

\end{definition}

\begin{proposition}[cf. {\cite[Proposition 1.2]{AL11}}]
Let $S$ be a smooth projective surface. 
Then we have the following inclusions:
\[ \SAmp(S) \subset \Mov(S) \subset \Nef(S) \subset \overline{\Eff(S)}  \,.\]
\end{proposition}

Let us recall the definition of Mori dream space.
\begin{definition}[{\cite[Definition 1.10]{HK00}}]
A variety $X$ is a  \textit{Mori dream space} if 
\begin{enumerate}
\item  $X$ is $\mathbb{Q}$-factorial variety 
with finitely generated Picard group  $\Pic(X)$, i.e. $h^1(X,\mathcal{O}_X)=0$,
\item the nef cone of $X$ is generated by finitely many semiample divisor classes, and 
\item there are finitely many birational maps $\phi_i : X \dashrightarrow X_i, 1 \leq i \leq m$ which are isomorphisms in codimension 1, $X_i$ are varieties satisfying (1), (2) and if $D$ is a movable divisor then there is an index $1 \leq i \leq m$ and a semiample divisor $D_i$ on $X_i$ such that $D=\phi^*_iD_i.$ 
\end{enumerate}
\end{definition}

Let us recall the definition of Cox ring.

\begin{definition}
Let $X$ be a normal projective $\mathbb{Q}$-factorial variety with finitely generated $\Cl(X)$. 
Let $\Gamma \subset \Cl(X)$ be a free Abelian group such that the inclusion map induces 
an isomorphism $\Gamma \otimes \mathbb{Q} \cong \Cl(X) \otimes \mathbb{Q}.$ 
Then a Cox ring of $X$ (associated to $\Gamma$) is a multi-graded ring defined as follows:
\[ \Cox(X) = \bigoplus_{D \in \Gamma}H^0(X,\mathcal{O}_X(D)). \]
\end{definition}

\begin{remark}
Although the above definition of a Cox ring depends on the choice of $\Gamma \subset \Cl(X)$, it is well-known (see \cite{HK00}) that its finite generation is independent of the choice. 
\end{remark}

\begin{theorem}[{\cite[Proposition 2.9]{HK00}}]
Let $X$ be a $\mathbb{Q}$-factorial variety with finitely generated Picard group  $\Pic(X)$.
Then $X$ is a Mori dream space if and only if $\Cox(X)$ is a finitely generated ring.
\end{theorem}

In dimension two there is a simpler criterion to decide whether a surface is a Mori dream space.

\begin{theorem}[{\cite[Theorem 2.5]{AHL10}}]\label{thmAHL10}
Let $S$ be a normal complete surface with finitely generated $\Cl(S)$. 
Then $S$ is a Mori dream space if and only if $\Eff(S)$ and $\Mov(S)$ are rational polyhedral cones and $\Mov(S)=\SAmp(S)$.
\end{theorem}

\begin{corollary}[{\cite[Corollary 2.6]{AHL10}}]\label{corAHL10}
Let $S$ be a  $\mathbb{Q}$-factorial projective surface with $q(S)=0$. Then $S$ is a Mori dream space if and only if $\Eff(S)$ is a rational polyhedral cone and $\Nef(S)=\SAmp(S).$
\end{corollary}

\begin{remark}
Since for surfaces the nef cone is dual to the closure of the effective cone, if $\Eff(S)$ is a rational polyhedral cone then  $\Nef(S)$ is also a rational polyhedral cone. In this case, it is sufficient to prove that extremal rays of $\Nef(S)$ are semiample to prove that $S$ is Mori dream space. 
\end{remark}

\subsection{Cones of fake quadrics}

Here, we will give a brief review about fake projective planes and fake quadrics. 

\begin{definition}
Let $S$ be a smooth minimal  surface of general type. We call $S$ a \textit{fake quadric} if the Hodge diamond of $S$ is the same as the Hodge diamond of a smooth quadric, i.e. 
$p_g(S)=q(S)=0$ and $K_S^2=8$.
\end{definition}

\begin{remark}
Fake quadrics have the second biggest $K^2$ among minimal surfaces of general type with $p_g=0.$
The maximal value $ K^2=9$ is achieved by the 
fake projective planes, i.e.~. smooth minimal  surface of general type
having the same as the Hodge diamond of the projective plane. 
 From Yau's theorem one can see that the univeral covering space of a fake projective plane is two-dimensional complex ball. From this description and investigation of fundamental groups of fake projective planes, it turns out that there are exactly 100 fake projective planes (cf. \cite{CS10, PY07, PY10}) and all of them provide examples of Mori dream surfaces of general type since they have Picard number one.
\end{remark}

Fake quadrics are interesting objects since they have Picard number  two (hence become the first nontrivial test cases) but their geometry is rather mysterious so far. Especially, we do not know how to classify fake quadrics. In \cite{BCP11}, the authors asked whether the universal cover of a fake quadric is always the product of two upper half planes or not and discussed how it is related to the famous question of Hirzebruch asking whether there is a surface of general type which is homeomorphic to $\mathbb{P}^1 \times \mathbb{P}^1$ or $\mathrm{Bl}_{\mathrm{pt}}\mathbb{P}^2$. All known examples of fake quadrics are uniformized by the product of two upper half planes. Let $S$ be  such a fake quadric. We call $S$ a \textit{reducible} fake quadric if it admits a finite unramified covering which is a product of two smooth projective curves, and \textit{irreducible} otherwise.

Reducible fake quadrics are surfaces isogenous to a product and we will discuss their classification in the next section. 
There has been a recent progress in the classification of irreducible fake quadrics (cf. \cite{LSV19}).
 However, we do not know much about the geometry of fake quadrics except a few cases. From this perspective, it is a very interesting question to know whether all fake quadrics are Mori dream spaces or not.

\begin{lemma}\label{NE-LE}

Let $S$ be a smooth projective surface with $q(S)=0$ and $D_1, D_2$ be two numerically equivalent effective divisors on $S$, then
\begin{itemize}
\item[a)] there exists a positive $m \in \mathbb{Z}$ such that $mD_1$ is linearly equivalent to $mD_2$.
\item[b)] if $D_1$, $D_2$ are effective and $D_1 \cap D_2 = \emptyset$. Then $D_1$ and $D_2$ are semiample divisors.
\end{itemize}
\end{lemma}
\begin{proof}
a) This follows directly from $\Pic^0(S)=0$.\\
b) Since $mD_1$ is linearly equivalent to $mD_2$ for some positive $m\in \mathbb Z$ and  $D_1 \cap D_2 = \emptyset$, then 
the linear system $|mD_1|$ has no base point, whence  $D_1$ and $D_2$ are semiample.
\end{proof}

\begin{proposition}\label{DivFQ} 
Let $S$ be a smooth projective surface with $q(S)=0$ and  Picard number 2. 
If  $D_1,D_2,D_3,D_4$  are four distinct effective irreducible divisors on $S$ such that 
\[D_1^2=D_2^2=D_3^2=D_4^2=0 \qquad  D_1.D_4=D_2.D_3=0\quad  D_1.D_2=D_1.D_3=D_4.D_2=D_4.D_3>0\,,\]
 then  $\Eff(S)=\Nef(S)=\SAmp(S)=\Cone\langle D_1, D_2 \rangle$. \\
In particular, $S$  is a Mori dream surface and does not have a negative curve.
\end{proposition}

\begin{proof}
Using the above intersection data and  since $\rho(S)=2$,
we can take $D_1, D_2$ as basis of $\Pic(S)_\mathbb R.$ Then we can see that $D_1\sim_{num}D_4$ and $D_2\sim_{num}D_3$ since they have the same intersection numbers with respect to this basis.
Moreover, since  $D_1$ does not meet $D_4$, they are semiample divisors, as well as $D_2$ and $D_3$,
whence $\Cone\langle D_1, D_2 \rangle \subset \SAmp(S) $.

Let $D=aD_1+b D_2$ be a nef divisor, then $0\leq D.D_1=aD_1.D_2$ and $0 \leq D.D_2=bD_1.D_2$ since $D_1, D_2$ are effective divisors. These inequalities imply that $\Nef(S) \subset \Cone\langle D_1, D_2 \rangle$, and we have the following inclusions.
\[\Cone\langle D_1, D_2 \rangle \subset \SAmp(S) \subset \Nef(S) \subset \Cone\langle D_1, D_2 \rangle\ \]
It follows that $\Nef(S)=\SAmp(S)=\Cone\langle D_1, D_2 \rangle$. 
By taking duals, we get  $\Eff(S) = \Cone\langle D_1, D_2 \rangle$ and we obtain the desired conclusion by Corollary \ref{corAHL10}.
\end{proof}

\section{Surfaces isogenous to a product}

An important class of fake quadrics are surfaces isogenous to a product with $p_g=0$.

\begin{definition}\label{def_iso}
A surface  $S$ is \textit{isogenous to a  product of curves}
if there exist compact Riemann surfaces $C_1, C_2$ of genus $g(C_i)\geq 2$  and 
a finite group  $G \leq \Aut(C_1 \times  C_2) $  acting freely on the product $C_1 \times C_2$ such that
$S=(C_1 \times C_2)/G$.
\end{definition}

In this situation every automorphism $f\in \Aut(C_1 \times  C_2)$ is either diagonal  $f(x,y)=(f_1(x), f_2(y))$, or (only if $C_1 \cong C_2$) of the form $f(x,y)=( f_2(y), f_1(x))$, for $ f_i\in \Aut(C_i), i=1,2$ (see \cite[Lemma 3.8]{Cat00}).
 Let $G^0:=G \cap (\mathrm{Aut}(C_1)\times \mathrm{Aut}(C_2))$ be the subgroup of $G$ of automorphisms acting diagonally.
 
\begin{remark} 	
By {\cite[Proposition 3.13]{Cat00}} we may assume (and we do!)  that 
the action is \textit{minimal}, i.e. both projections $G^0\to \Aut(C_i),\, i=1,2$ are injective.

\end{remark}

If $G^0=G$ the action of $G$ is  called \textit{unmixed}, otherwise \textit{mixed}. In the latter case $C_1\cong C_2=:C$, $G^0$ is an index 2 subgroup of $G$ (indeed $\Aut(C\times C)= \Aut(C) \rtimes \mathbb Z_2$ by  \cite[Corollary 3.9]{Cat00}), 
and we identify  $G^0<\Aut(C) \times \Aut(C)$ with its projection onto the first factor.

\begin{remark} 	
i) It follows from the definition that a surface  isogenous to a product  $S$  is a smooth minimal surface of  general type
with invariants (see \cite[Theorem 3.4]{Cat00})
\[ \chi(S)= \frac{(g(C_1)-1)(g(C_2)-2)}{|G|}\,,\quad K^2_S=8\chi(S)\,,\quad e(S)=4\chi(S) \,, \]
and $q(S)=g(C_1/G)+ g(C_2/G)$ in the unmixed case, and  $q(S)=g(C/G)$ in the mixed case.

ii) Surfaces isogenous to a product have been studied by several authors, and we have nowadays a complete classification for  $p_g=q$, i.e. $K^2=8$ (see \cite{BCG08,Frap13,CP09,pen11}). 

iii) In \cite[Lemma 3.6]{KL19} Keum and the second author showed that surfaces isogenous to a product of unmixed type with $p_g=0$ are Mori dream spaces, giving a description of their effective, nef and semiample cones.
 \end{remark}

From now on we focus on the mixed case. We have the following description of minimal mixed actions:
\begin{theorem}[{\cite[Proposition 3.16]{Cat00}}]
\label{thmaction} 
Let $C$ be a compact Riemann surface of genus $g(C)\geq 2$, and let $G$ be a finite subgroup of
$\Aut(C\times C)$ whose action is minimal and mixed. Fix $\tau'
\in G\setminus G^0$; it determines an element $\tau:=\tau'^2\in G^0$ and $\varphi\in\Aut(G^0)$ defined by $\varphi(h):=\tau'h\tau'^{-1}$.
Then, up to a coordinate change, $G$ acts on $C \times C$ as follows:
\begin{equation}
\label{action}
\begin{split}
g(x,y)=(gx,\varphi(g)y)\\
\tau' g(x,y)=(\varphi(g)y,\tau gx)
\end{split}
\qquad \text{for }g\in G^0\,.
\end{equation}

Conversely, for every finite subgroup $G^0 <\Aut(C),$  extension group $G$  of degree $2$ of $G^0$, fixed $\tau'\in G \setminus G^0$ and $\tau, \varphi$ defined as above, \eqref{action} defines a minimal mixed action on $C\times C$.
\end{theorem}

The  description of surfaces isogenous to a product  is accomplished through the theory 
  of Galois coverings between projective curves (cf. \cite[Section III.3, III.4]{Mir}).

\begin{definition}\label{gv}
Given integers  $g'\geq 0$, $m_1, \ldots, m_r > 1$  and   a finite group $H$, a \textit{generating vector} for $H$ of type $[g';m_1,\ldots ,m_r]$ is a $(2g'+r)$-tuple of
elements of $H$:
\[V:=(d_1,e_1,\ldots, d_{g'},e_{g'};h_1, \ldots, h_r)\]
such that $V$ generates $H$, $\prod_{i=1}^{g'}[d_i,e_i]\cdot h_1\cdot h_2\cdots h_r=1$ and $\ord(h_i)=m_i$.
\end{definition}

\begin{theorem}[Riemann Existence Theorem]\label{RET}
A finite group $H$ acts as a group of automorphisms of
some compact Riemann surface $C$ of genus $g(C)$, 
if and only if  there exists a generating vector $V:=(d_1,e_1, \ldots , d_{g'},e_{g'};h_1, \ldots , h_r)$ 
for  $H$ of type $[g';m_1,\ldots,m_r]$,
such that  the  Hurwitz formula holds:
\[
2g(C) -2 = |H| \bigg( 2g'-2+\sum_{i=1}^r \frac{m_i - 1}{m_i} \bigg). 
\]
In  this case $g'$ is the genus of the quotient Riemann surface $C' := C/H$
and the $H$-cover $C\to C'$ is branched over $r$ points $\{x_1, \ldots , x_r\}$ with
branching indices $m_1, \ldots ,m_r$, respectively.
Moreover,  the \textit{stabilizer set of $V$}, defined as 
\[\Sigma_V:= \bigcup_{g\in H}\bigcup_{j\in \mathbb Z}\bigcup_{i=1}^r \{ g\cdot h_i^j\cdot g^{-1}\}\,,\]
coincides with the subset of $H$ consisting of the automorphisms of $C$ having some fixed points.
\end{theorem}

According to \cite{BCG08}, a surface isogenous to a product of mixed type $(C\times C)/G$ determines and is determined 
(using the Riemann Existence Theorem) by:
a finite group $G$,  an index $2$ subgroup $G^0$, a  generating vector $V$ for $G^0$  of type $[g';m_1,\ldots,m_r]$, 
such that: i) $\Sigma_V\cap\varphi(\Sigma_V)=\{1_G\}$ (i.e. no isolated fixed points);
ii)  there are no $g\in G\setminus G^0$, such that $ g^2\in\Sigma_V$ (i.e. no fixed curves);
and $r$ points on $C':=C/G^0$. Hence each surface isogenous to a product of mixed type depends
on $3g'-3+r$ parameters and to classify  surfaces isogenous to a product  is equivalent to classify groups with the right properties.
This algebraic description has been used in \cite{BCG08,Frap13} to  obtain the following classification.

\begin{theorem}[\cite{BCG08,Frap13}]\label{thm:classification}
Let $S:=(C \times C)/G$ be a surface isogenous to a product of mixed type with $p_g=0$. 
Then $S$ belongs to one of the 5 families  in Table \ref{q0}. 

Moreover, each entry in the table gives an irreducible connected
component of dimension $D$ in the moduli space of minimal smooth complex projective surfaces with $\chi=1$ and $K^2=8$.

\begin{table}[h!]
 \begin{tabular}{c|c|c|c|c|c|c|c}
$G$ & $Id(G)$ & $G^0$ & $Id(G^0)$ & $g(C)$ & Type & $H_1(S,\mathbb Z)$&$D$\\
\hline
 $D_{2,8,5}\rtimes\mathbb{Z}_2^2$ &  $G(64,92)$& $\mathbb{Z}_2^2\times D_4$ &  $G(32,46)$&$9$&$[0;2^5]$& $ \mathbb Z_2^3\times\mathbb Z_8$ &$2$\\
$G(256,3679)$ &$G(256,3679)$& $(\mathbb{Z}_2^3\rtimes \mathbb{Z}_4)\rtimes \mathbb{Z}_4$ &  $G(128,36)$&$17$&$[0;4^3]$& $\mathbb Z_2^2\times\mathbb Z_4^2$ &$0$\\
$G(256,3678)$& $G(256,3678)$& $(\mathbb{Z}_2^3\rtimes \mathbb{Z}_4)\rtimes \mathbb{Z}_4$ &  $G(128,36)$&$17$&$[0;4^3]$&$\mathbb Z_2^2\times\mathbb Z_4^2$ &$0$\\
 $G(256,3678)$& $G(256,3678)$& $(\mathbb{Z}_2^3\rtimes \mathbb{Z}_4)\rtimes \mathbb{Z}_4$ & $G(128,36)$&$17$&$[0;4^3]$&$ \mathbb Z_2^4\times\mathbb Z_4$&$0$ \\
$G(256,3678)$& $G(256,3678)$& $(\mathbb{Z}_2^3\rtimes \mathbb{Z}_4)\rtimes \mathbb{Z}_4$ &  $G(128,36)$&$17$&$[0;4^3]$&$ \mathbb Z_4^3$&$0$
\end{tabular} \caption{Surfaces isogenous to a product of mixed type with $p_g=q=0$}	\label{q0}
\end{table}
 \end{theorem}

In Table \ref{q0} we use the following notation: columns  $Id(G)$ and $Id(G^0)$  report the MAGMA (\cite{MAGMA}) identifier of the groups $G$ and $G^0$: $G(a,b)$ denotes the $b^{th}$ group of order $a$ in the database of Small Groups.
The column Type gives the abbreviated type of the generating vector for  $G^0$, e.g. $[0; 2^5]$ stands for
$[0; 2,2,2,2,2]$.

\subsection{Extra Automorphisms} \label{E_A}
Let us consider families 2-5 in Table \ref{q0}.
In these  cases   $S$ is uniquely determined by $G$ and  a  generating vector for $G^0$ of  type $[0;4^3]$, because these last 2 data determine a unique (up to isomorphism)  curve $C$,
being  the covering $c_1:C\to C/G^0=\mathbb P^1$ branched over 3 points.
We claim that $\Aut(C)$ is much bigger than $G^0$.

We may assume that $c_1$ is  branched over $P_1:=[1:1]$, $P_2:=[1:\xi]$, $P_3:=[1:\xi^2]$, 
($\xi:=\exp(\frac{2\pi i}{3})$).
The rotation $\sigma\colon \mathbb P^1\to \mathbb P^1$, $\sigma[x:y]=[x:\xi y]$ has 2 fixed points: $\infty=[0:1]$ and $0=[1:0]$, and permutes the $P_i$. 
Let $c_2:=\mathbb P^1 \to \mathbb P^1, [x:y]\mapsto [x^3:y^3]$ be the quotient map by  $\mathbb Z_3=\langle \sigma\rangle$.
The composition $c_2\circ c_1:C\to \mathbb P^1$ is then branched over 3 points with branching indices $3,3,4$.
We then have a further involution on $\mathbb P^1$: $\rho:\mathbb P^1\to \mathbb P^1$, $\rho[x:y]=[y:x]$, 
switching the images of  $0$ and $\infty$ and having $[1:1]$ (image of the $P_i$) and $[1:-1]$ as fixed points.
Let $c_3:=\mathbb P^1 \to \mathbb P^1, [x:y]\mapsto [xy: x^2+y^2]$ be the quotient map by  $\mathbb Z_2=\langle \rho\rangle$.
The composition $c:=c_3\circ c_2\circ c_1\colon C\to \mathbb P^1$ is a degree 768 covering branched over 3 points with branching  indices $2,3,8$.
Next we show that $c\colon C\to \mathbb P^1$ is Galois.

The supporting MAGMA script  available at 
\url{http://www.staff.uni-bayreuth.de/~bt301744/} 
shows that there exists a unique group $H$ of order $768$ having a generating vector of type $[0; 2,3,8]$,
namely $H:=G(768, 1085341)$. This group has $G^0$ as normal subgroup, more precisely $G^0=[H',H']$, where  $H':=[H,H]$ 
is the commutator subgroup of $H$. We now note that if $(a,b,c)$ is a generating  vector of type $[0; 2,3,8]$ for $H$,
then $(d,e,f):=(aba^{-1},b,c^2)$ is a generating  vector of type $[0; 3,3,4]$ for $H'=G(384,4)$ 
 and $(efe^{-1},e^2fe^{-2},f)$ is a generating  vector of type $[0; 4^3]$ for $G^0$. 
The mentioned MAGMA script shows that
for all 4 families there exists a generating  vector $(a,b,c)$ of type $[0;2,3,8]$ for $H$ whose induced generating  vector for $G^0$ is the given one,
whence  the covering $c\colon C\to \mathbb P^1$ is $H$-Galois, and the  embedding $H < \Aut(C)$ is
given by the generating  vector $(a,b,c)$.

\section{$G$-equivariant geometry of $C \times C$}

In this section we explain how to construct effective divisors on a surface isogenous to a product of mixed type.

Let $S:=(C\times C )/G$ be a surface isogenous to a product of mixed type and 
let $\eta: C\times C \to S$  be the quotient map.
For  $f\in \Aut(C)$ we denote by $\Gamma_f $ its graph: $\Gamma_f:=\{(x,fx )\in C \times C\mid x \in C\}$. 
\begin{lemma}\label{act_Gamma}
Let $h$ be in $G^0$, then $h(\Gamma_f)= \Gamma_{\varphi(h)fh^{-1}}$ and $ \tau'h(\Gamma_f)= \Gamma_{\tau h  f^{-1}\varphi(h^{-1})}$.
\end{lemma}
\begin{proof}
The statement is a straightforward computation using the formula \eqref{action}:
\begin{eqnarray*}
h(x,fx)=(hx, \varphi(h) f x)\stackrel{y=hx}{=} 
(y, \varphi(h)fh^{-1}  y)  &\quad \Longrightarrow\quad &
h(\Gamma_f)= \Gamma_{\varphi(h)fh^{-1}}\\
\tau'h(x,fx)=(\varphi(h) f x, \tau h x)\stackrel{z=\varphi(h) fx}{=} 
(z, \tau h  f^{-1}\varphi(h^{-1}) z ) & \quad \Longrightarrow\quad &
\tau'h(\Gamma_f)= \Gamma_{\tau h  f^{-1}\varphi(h^{-1})}
\end{eqnarray*}
\end{proof}

Let $H$ be a subgroup of $\Aut(C)$ containing $G^0$: $G^0<H< \Aut(C)$.
By Lemma \ref{act_Gamma}, the $G$ action on $C\times C$ induces an action on the set $\{\Gamma_f \mid f \in H\}$. Let $f\in H$ and consider the orbit of $\Gamma_f$: $\{\gamma(\Gamma_{f}) \mid \gamma \in G\}= \{\Gamma_{f_1},\ldots, \Gamma_{f_n}\}$, where  $n$ is the index of the stabilizer subgroup $\{\gamma \in G \mid \gamma(\Gamma_{f})=\Gamma_{f}\}$ of $\Gamma_f$.
The sum of these $n$ effective divisors
\[\tilde D:= \left(\sum_{\gamma \in G}\gamma(\Gamma_{f})\right)_{red}=\Gamma_{f_1}+\ldots+\Gamma_{f_n}\,,\]
is  $G$-invariant, so it yields an effective 
 divisor $D:=\eta_*(\tilde D)_{red}$ on $S$, and
$\tilde D= \eta^*D$.
Here the subscript ``red'' means that we consider the reduced structure of that effective divisor, i.e.~ every prime component is taken with multiplicity one.

 We call  $D$ the  \textit{orbit divisor} induced by  $f$.
Note that $D$ is irreducible, as the image of an irreducible divisor: $D= \eta(\Gamma_f)$.

\begin{lemma} Let $D:=\eta_*(\Gamma_{f_1}+\ldots+\Gamma_{f_n})_{red}$ and  $D':=\eta_*(\Gamma_{f'_1}+\ldots+\Gamma_{f'_m})_{red}$ be two distinct orbit divisors on $S$ induced by $f$ and  $f'$ respectively.
 Then
\begin{eqnarray}
D.D'&=& \frac{1}{|G|} \sum_{i=1}^{n}  \sum_{j=1}^{m} \Gamma_{f_i} . \Gamma_{f'_j}\\
D^2 &=& \frac{-2(g(C)-1)n}{|G|}  +\frac{2}{|G|}\sum_{1\leq i<j\leq n} \Gamma_{f_i}. \Gamma_{f_j}  \\
K_S.D&= &\frac{ 4(g(C)-1) }{|G|}n \,.
\end{eqnarray}
\end{lemma}

\begin{proof}
Using the projection formula we get immediately the first formula and the equation
\[ D^2 =\frac{1}{|G|} (\eta^*D)^2= \frac{1}{|G|} \left(\sum_{i=1}^n \Gamma_{f_i}\right)^2=
 \frac{1}{|G|} \left(\sum_{i=1}^n \Gamma_{f_i}^2 +2\sum_{1\leq i<j\leq n} \Gamma_{f_i}. \Gamma_{f_j} \right)\,.
\]
The values $ \Gamma_{f_i}^2$ can be easily computed using the
adjunction formula $K_{C\times C}. \Gamma_{f_i}+ \Gamma_{f_i}^2 =2(g(C)-1)$:
\[\Gamma_{f_i}^2 =  -2(g(C)-1) \,,\]
since
\[K_{C\times C}. \Gamma_{f_i}= 2(g(C)-1)(F_1+F_2). \Gamma_{f_i}= 4(g(C)-1)\,,\]
where $F_1,F_2$ are the general fiber of the projection respectively onto  the first and  the
second coordinate.

Since $\eta$ is \'etale, $K_{C\times C}= \eta^* K_S$ and the third formula follows.
\end{proof}

\begin{lemma}\label{Int_transv}
Let $f_1\neq  f_2 \in \Aut(C)$, 
then $\Gamma_{f_1}$ and $\Gamma_{f_2}$ intersect transversally.
\end{lemma}

\begin{proof}
The curves intersect in points of the form $(x, f_1x)= (x,f_2x)$, i.e, $f_1^{-1} f_2\in \Stab(x)$.
Up to acting by the automorphism $\alpha \colon C\times C \to  C\times C \,, \ (u,v) \mapsto (u, f_1^{-1} v)$,
we can assume that the curves have the form $\{(u,u)\}$, $\{(u, f_1^{-1} f_2 u)\}$ and that the intersection 
takes place at the point $(x_0,x_0)= (x_0, f_1^{-1} f_2 x_0)$.
By Cartan's Lemma, the local behaviour of $f_1^{-1} f_2$ is the same as that of the linearization,
since $ f_1^{-1} f_2\in \Stab(x_0)$ is not the identity, then $\Jac(f_1^{-1} f_2)_{x_0}\neq 1$,
and so the curves intersect transversally.
\end{proof}

By Lemma \ref{Int_transv}, to determine $\Gamma_{f_1}.\Gamma_{f_2}= |\Gamma_{f_1}\cap\Gamma_{f_2}|$, it is enough
to count the number of points on $C$ fixed by  the automorphism $f_1^{-1} f_2 \in H$.
This can be done using the Riemann Existence Theorem  for the covering
  $C\to C/H$. Here we use the notation of Theorem \ref{RET} and   we denote by 
$K_j:=\langle h_j\rangle$ the stabilizer of a point in the fiber over $x_j \in C/H$,  and by 
$\textbf 1_A(\cdot):G\to \{0,1\}$   the indicator function of the subset $A\subseteq G$, i.e.~$1_A(g)=1$ if $g\in A$, and  $1_A(g)=0$ else.

\begin{lemma}\label{fix}
The number of fixed points of the non-trivial automorphism $ f\in H$ is:
\[|\Fix( f)|=\sum_{j=1}^r \frac{1}{m_j}\sum_{g\in H} \textbf 1_{gK_jg^{-1}}( f)\,.\]
\end{lemma}	

\begin{proof} By the (proof of) Riemann Existence Theorem, the fiber of $x_j \in C/H$ is in
1-1 correspondence with the cosets $\{gK_j\}_{g\in H}$.
So  $f\in H$ fixes the point $gK_j (=:y)$ if and only if $ f\in  gK_jg^{-1} (= \Stab(y))$. Since every coset has $|K_j|=m_j$ representatives, the sum 
$\sum_{g\in H} \textbf 1_{gK_jg^{-1}}( f)$ equals $m_j$-times the number of fixed points over $x_j$.
\end{proof}

\begin{theorem}\label{mainThm}
Let $S$ be a surface isogenous to a product of mixed type with $p_g=0$.
 Then $\Eff(S)=\Nef(S)=\SAmp(S)$.
 In particular $S$ is a Mori dream surface and does not have a negative curve.
\end{theorem}

\begin{proof}
We prove the statement  case by case, applying  the machinery explained in the previous section to a suitable automorphism group $H$:  $G^0 <H < \Aut(C)$.

\textbf{Family 1}: The surfaces $S$  are determined by  the group  $G:=G(64,92)$, 
the subgroup $G^0:= G(32,46)$ and a generating vector for $G^0$ of  type $[0;2^5]$  giving a  curve $C$.
As automorphism group $H$ we consider the group $G^0$.

The MAGMA script in Appendix \ref{script} shows that
the 32 curves $\Gamma_f$, $f \in H$ induce 4 distinct orbit divisors  on $S$, whose intersection products are:
\[K_X. D_i=4 \,,\quad    D_i^2=D_1.D_4=  D_2.D_3=0\,, \text{ and  } \ D_1.D_2=  D_1.D_3 =  D_2.D_4=  D_3.D_4=4\,.\]
By Lemma \ref{DivFQ}, $\Eff(S)=\Nef(S)=\SAmp(S)=\Cone \langle D_1, D_2 \rangle$ and $S$ is a Mori dream space.

\textbf{Families 2-5}:
In these cases $S$ is uniquely determined by $G$ and  a  generating vector for $G^0$ of  type $[0;4^3]$, because these last 2 data determine a unique (up to isomorphism) curve $C$.
By discussions in Section \ref{E_A},   
the group $H:=G(768, 1085341)$ acts as automorphism group on $C$ 
and contains $G^0$ as index 6 normal subgroup.

In Case 2 (in the other 3 cases we have an analogous output) the supporting MAGMA script, 
available at \url{http://www.staff.uni-bayreuth.de/~bt301744/}, shows that the elements of $H$ induce 15 distinct orbit divisors  on $S$,
whose intersection products  show that:\\
1) $D_2, D_{9}$ and $D_{13}$ are numerically  equivalent and $D_2^2=0$; \\
2) $D_7, D_{14}$ and $D_{15}$ are numerically equivalent, $D_7^2=0$,  and $D_2.D_7=16$;\\
3) $D_1, D_8, D_{10}$ and $D_{12}$ are numerically equivalent  to $(D_2 +D_7)/2$;\\
4) $D_5, D_6$,  and $D_{11}$ are numerically equivalent  to $D_2 +D_7$; \\
5) $D_3$ and $D_4$ are numerically equivalent  to $2(D_2 +D_7)$.

Applying  Lemma \ref{DivFQ} to $D_2, D_{7},D_9,D_{14}$, we get 
$\Eff(S)=\Nef(S)=\SAmp(S)=\Cone\langle D_2, D_7 \rangle$ and $S$ is a Mori dream space.

\end{proof}

\appendix
\section{Magma script}\label{script}

This script may be run at \url{http://magma.maths.usyd.edu.au/calc/}.

\begin{code_magma}
// This auxiliary function counts the number of  fix-points of f in H < Aut(C) using Lemma 3.4.
// seq is the generating vector associated to the covering p:C->C/H=P^1.
// Note that the fixed points of f are ramification points of the map p.
CountingIntersections:=function(f, seq, H) 
fix:=0;
for j in [1..#seq] do  // running over j corresponds to run over the branch points x_j of p
	c:=0; K:=sub<H|seq[j]>; 
	// c is the "local" counter of  points fixed by f  in the  fiber of x_j;
	// K is the subgroup of H  generated by seq[j], stabilizer of a point in the  fiber of x_j.
	// The stabilizers of the other points in this fiber are of the form gKg^{-1}, g in H and
	// f is in gKg^{-1} if and only if it stabilizes the corresponding point.
		for g in H do 
			if f in  {g*k*g^-1: k in K}  then c:=c+1;  
		end if; end for; 
	fix:=fix+(c/#K); 
	// #K=|K| elements g in H give the same set gKg^{-1}, so we have to divide the local counter 
	// c by #K  to obtain the number of point fixed by f in the  fiber of x_j.
end for;
return fix; 
end function;

// The algebraic data defining the family are:
G:=SmallGroup(64,92); 
ram:=[ G.2 * G.6, G.2 * G.3 * G.4, G.2 * G.3 * G.5 * G.6, G.2 * G.5, G.4 ];
G0:=sub<G|ram>;  
// G0(=G^0) and ram are the algebraic data describing the covering C->C/G0.
t:=Rep({x: x in G | x notin G0});
// t=\tau' is an element of G not in G0 which determine the mixed action of G on CxC.
genus_1:=8; //g(C)-1

// Here start the main routine.

// Step 1. We construct all orbit divisors induced by elements of G^0.
L:=[ ]; P:=Set(G0); 
while not IsEmpty(P) do 
// 1a. As long as P it is not empty we pick an element f from it.
	f:=Rep(P); 
// 1b. We construct the orbit Gamma of \Gamma_f under the G-action given by Lemma 3.1, 
// and we store in the list L (each curve is represented by the corresponding element in G0).
	Gamma:={(t*h*t^-1)*f*(h^-1) : h in G0} join {(t^2*h)*f^-1*(t*h^-1*t^-1) : h in G0}; 
	Append(~L, Gamma); 
// 1c. We remove the elements in Gamma from P and we go back to step 1a.	
	P:=P diff Gamma;
end while; 
// Each entry in L consists in an orbit of the G-action, i.e. corresponds to an orbit divisor D_i.
printf "There are 

// Step 2. For each orbit divisor D_i ("=L[i]")  we compute D_i^2 and D_i.K_S   using Lemma 3.2.
for i in [1..#L] do 
// 2a. We set n as the cardinality of L[i] and initialize the value of self  to -2*n*(g(C)-1).
	Gamma:=L[i]; n:=#Gamma; 
	self:= -2*n* genus_1; 
// 2b.	For each pair (g1,g2) of  distinct elements in L[i]  we determine the intersection product
// \Gamma_{g1}.\Gamma_{g2} = the number of fix-points of the element g1^-1*g2 in G0, 
// and we increase self accordingly (see equation (3.2)).
// Note that both pairs (g1,g2) and (g2,g1) are  considered.
	for g1 in Gamma do  	for g2 in Gamma diff {g1} do 
		int:=CountingIntersections(g1^-1*g2,ram,G0);
		self:=self+int;
	end for; end for;
// 2c. D_i^2 and D_i.K_S are then given by self/|G| and 4*(g(C)-1)*n/|G|.
	printf "D_
end for;

// Step 3. Finally, for i<j we compute D_i.D_j using Lemma 3.2
for i in [1..#L] do for j in [i+1 .. #L] do  
	prod:=0;
// 3a. For each pair  (g1,g2) with g1 in L[i] and g2 in L[j]  we determine the intersection product
// \Gamma_{g1}.\Gamma_{g2} = the number of fix-points of the element g1^-1*g2 in G0, 
// and we increase prod accordingly (see equation (3.1)).
	for g1 in L[i] do for g2 in L[j] do 
		int:=CountingIntersections(g1^-1*g2,ram,G0);
		prod:=prod+int;
	end for; end for;
// 3b. D_i.D_ is then given by prod/|G|.
	printf "D_
end for; end for;

\end{code_magma}

\bibliographystyle{alpha}

\end{document}